\documentclass{amsart}
\usepackage{graphicx}
\usepackage{ulem}
\usepackage{bm}
\usepackage{url}
\usepackage{color}
\usepackage{amssymb, amsmath, amsthm}
\usepackage{bm}
\theoremstyle{plain}
\usepackage{graphicx}
\usepackage[aboveskip=-5pt,position=bottom]{caption}
\usepackage[numbers]{natbib}
\usepackage{enumerate,color}

\newtheorem{lemma}{Lemma}[section]

\newtheorem{theorem}[lemma]{Theorem}
\newtheorem{cor}[lemma]{Corollary}
\newtheorem{prop}[lemma]{Proposition}
\newtheorem{exam}[lemma]{\normalfont \scshape
 Example}
\newtheorem{rem}[lemma]{\normalfont \scshape Remark}

\newcommand{\R}{\mathbb{R}}
\newcommand{\N}{\mathbb{N}}

\newcommand{\norm}[1]{\left\Vert#1\right\Vert}

\newcommand{\abs}[1]{\left\vert#1\right\vert}
\newcommand{\set}[1]{\left\{#1\right\}}

\newcommand{\bfx}{\bm{x}}
\newcommand{\bfzero}{\bm{0}}
\newcommand{\bfinfty}{\bm{\infty}}

\newcommand{\bfone}{\bm{1}}
\newcommand{\bfa}{\bm{a}}

\newcommand{\bfU}{\bm{U}}
\newcommand{\bfu}{\bm{u}}

\newcommand{\bfV}{\bm{V}}

\newcommand{\bfX}{\bm{X}}
\newcommand{\bfY}{\bm{Y}}
\newcommand{\bfy}{\bm{y}}
\newcommand{\bfZ}{\bm{Z}}

\newcommand{\bfeta}{\bm{\eta}}

\newcommand{\bfxi}{\bm{\xi}}

\allowdisplaybreaks[4]

\newcommand{\RR}{\mathbb{R}}

\newcommand{\EE}{{\rm E}}

\newcommand{\ind}{\mathbf{1}}

\newcommand{\QED}

\begin{document}

\thispagestyle{empty}
\title[The max-characteristic function]{An offspring of multivariate extreme value theory: the max-characteristic function}
\author{Michael Falk$^{(1)}$ \& Gilles Stupfler$^{(2)}$}
\let\thefootnote\relax\footnote{\small $^{(1)}$ Universit\"at W\"urzburg, 97074 W\"urzburg, Germany \\
$^{(2)}$ School of Mathematical Sciences, The University of Nottingham, University Park, \\ Nottingham NG7 2RD, United Kingdom}

\email{michael.falk@uni-wuerzburg.de, Gilles.Stupfler@nottingham.ac.uk}

\begin{abstract}
This paper introduces max-characteristic functions (max-CFs), which are an offspring of multivariate extreme-value theory. A max-CF characterizes the distribution of a random vector in $\R^d$, whose components are nonnegative and have finite expectation. Pointwise convergence of max-CFs is shown to be equivalent to convergence with respect to the Wasserstein metric. The space of max-CFs is not closed in the sense of pointwise convergence. An inversion formula for max-CFs is established.
\end{abstract}



\maketitle

\section{Introduction}
\label{intro}

Multivariate extreme-value theory (MEVT) is the proper toolbox for analyzing several extremal events simul\-taneous\-ly. Its practical relevance in particular for risk assessment is, consequently, obvious. But on the other hand MEVT is by no means easy to access; its key results are formulated in a measure theoretic setup; a common thread is not visible.

Writing the `angular measure' in MEVT in terms of a random vector, however, provides the missing common thread: Every result in MEVT, every relevant probability distribution, be it a max-stable one or a generalized Pareto distribution, every relevant copula, every tail dependence coefficient etc. can be formulated using a particular kind of norm on multivariate Euclidean space, called $D$-norm; see below. For a summary of MEVT and $D$-norms we refer to \citet{fahure10, aulbf11, aulfaho11b,aulfaho11, aulfahozo14, aulfazo14, falk15}. For a review of copulas in the context of extreme-value theory, see, e.g., \citet{gennes12}.

A norm $\norm\cdot_D$ on $\R^d$ is a \textit{$D$-norm}, if there exists a random vector (rv) $\bfZ=(Z_1,\dots,Z_d)$ with $Z_i\ge 0$, ${\rm E}(Z_i)=1$, $1\le i\le d$, such that
\begin{equation*}
\norm{\bfx}_D={\rm E}\left\{\max_{1\le i\le d}\left(\abs{x_i}Z_i\right)\right\},\qquad \bfx=(x_1,\dots,x_d)\in\R^d.
\end{equation*}
In this case the rv $\bfZ$ is called \textit{generator} of $\norm\cdot_D$. Here is a list of $D$-norms and their generators:
\begin{itemize}
\item $\norm{\bfx}_\infty = \max_{1\le i\le d}\abs{x_i}$ is
generated by $\bfZ=(1,\dots,1)$,
\item $\norm{\bfx}_1=\sum_{i=1}^d\abs{x_i}$ is
generated by $\bfZ=$ random permutation of $(d,0,\dots,0)\in\R^d$ with equal probability $1/d$,
\item $\norm{\bfx}_\lambda=\left(\sum_{i=1}^d\abs{x_i}^\lambda\right)^{1/\lambda}$,
   $1<\lambda<\infty$.  Let $X_1,\dots,X_d$ be independent and identically Fr\'{e}chet-distributed random variables,  i.e., $\Pr(X_i\le x)=$ $\exp(-x^{-\lambda})$, $x>0$, $\lambda>1$. Then $\bfZ=(Z_1,\dots,Z_d)$ with
    \begin{equation*}\label{eq:frechet_generator}
    Z_i=\frac{X_i}{\Gamma(1-1/\lambda)},\quad i=1,\dots,d,
    \end{equation*}
    generates $\norm\cdot_{\lambda}$. By $\Gamma(p)=\int_0^\infty x^{p-1}e^{-x}\,dx$, $p>0$, we denote the usual Gamma function.
\end{itemize}
$D$-norms are a powerful tool when analyzing dependence in MEVT. The first letter of the word ``dependence'' is, therefore, the reason for the index $D$.

The generator of a $D$-norm is not uniquely determined, even its distribution is not. Let, for example, $X\ge 0$ be a random variable with ${\rm E}(X)=1$ and put $\bfZ=(X,\dots,X)$. Then $\bfZ$ generates $\norm\cdot_\infty$ as well. However, we can, given a generator $\bfZ$ of a $D$-norm, design a $D$-norm in a simple fashion so that it characterizes the distribution of $\bfZ$: consider the $D$-norm on $\R^{d+1}$
$$
(t,\bfx)\mapsto {\rm E}\left\{\max(\abs t,\abs{x_1}Z_1,\dots ,\abs{x_d}Z_d)\right\}.
$$
Then it turns out that the knowledge of this $D$-norm fully identifies the distribution of $\bfZ$; it is actually enough to know this $D$-norm when $t=1$, as Lemma~\ref{lem:uniqueness_of_max-cf} below shows, and this shall be the basis for our definition of a max-characteristic function.

\begin{lemma}\label{lem:uniqueness_of_max-cf}
	Let	$ \bfX = (X_1,  \dots, X_d) \geq \bfzero$, $\bfY = (Y_1,  \dots, Y_d) \geq \bfzero$ be random vectors with
	$ {\rm E}(X_i), {\rm E}(Y_i) < \infty$ for all $i \in \{ 1, \ldots , d\}$.
	If we have for each $\bfx > \bfzero \in \R^d$
\[
{\rm E}\left\{\max(1,x_1X_1,\dots ,x_dX_d)\right\} = {\rm E}\left\{\max(1,x_1Y_1, \dots, x_dY_d)\right\},
\]
then $\bfX=_d\bfY$, where ``$=_d$'' denotes equality in distribution.  	
	\end{lemma}

\begin{proof}
Fubini's theorem implies ${\rm E}(X)=\int_0^\infty \Pr(X>t)\,dt$ for any random variable $X\ge 0$. consequently, we
	have for $\bfx > \bfzero$ and $c > 0$
	\begin{align*}
{\rm E}\left\{\max\left(1,\frac{X_1}{cx_1},\dots ,\frac{X_d}{cx_d}\right)\right\}
	&= \int_{0}^{\infty} 1-\Pr\left\{\max\left(1,\frac{X_1}{cx_1},\dots ,\frac{X_d}{cx_d}\right) \leq t \right\} \, dt\\
	&=	\int_{0}^{\infty}1-\Pr(1 \leq t, X_i \leq tcx_i,\, 1 \leq i \leq d) \, dt\\
	&= 1+\int_{1}^{\infty}1-\Pr\left(X_i \leq tcx_i, \,1 \leq i \leq d  \right) \, dt.
	\end{align*}
	The substitution $t \mapsto t/c $ yields that the right-hand side above equals
	\[
	1+\frac{1}{c}\int_{c}^{\infty}1-\Pr(X_i\leq tx_i, 1 \leq i \leq d) \, dt.
	\]
Repeating the preceding arguments with $Y_i$ in place of $X_i$, we obtain for all $c>0$ from the assumption the equality
$$\int_{c}^{\infty}1-\Pr(X_i\leq tx_i, 1 \leq i \leq d) \, dt
= \int_{c}^{\infty}1-\Pr(Y_i\leq tx_i, 1 \leq i \leq d)\, dt.
$$
	Taking right derivatives with respect to $c$ we obtain for $c>0$
	\begin{align*}
	1-\Pr(X_i\leq cx_i,\, 1 \leq i \leq d) = 1 - \Pr(Y_i \leq cx_i,\, 1\leq i \leq d),
	\end{align*}
	and, thus, the assertion.
\end{proof}

Let $\bfZ=(Z_1,\dots,Z_d)$ be a random vector, whose components are nonnegative and integrable. Then we call
\[
\varphi_{\bfZ}(\bfx)={\rm E}\left\{\max\left(1,x_1Z_1,\dots,x_dZ_d\right)\right\}, \quad \bfx=(x_1,\dots,x_d)\ge\bfzero\in\R^d,
\]
the \textit{max-characteristic function} (max-CF) pertaining to $\bfZ$. Lemma \ref{lem:uniqueness_of_max-cf} shows that the distribution of a nonnegative and integrable random vector $\bfZ$ is uniquely determined by its max-CF.

Some obvious properties of $\varphi_{\bfZ}$ are $\varphi_{\bfZ}(\bfzero)=1$, $\varphi_{\bfZ}(\bfx)\ge 1$ for all $\bfx$ and
\[
\varphi_{\bfZ}(r\bfx)\begin{cases}
\le r \varphi_{\bfZ}(\bfx)& \mbox{if } r\ge 1,\\
\ge r \varphi_{\bfZ}(\bfx)& \mbox{if } 0 \le r\le 1.
\end{cases}
\]
It is straightforward to show that any max-CF is a convex function and, thus, it is continuous and almost everywhere differentiable; besides, its derivative from the right exists everywhere. This fact will be used in Section \ref{subsec:an_inversion_formula}, where we will establish an inversion formula for max-CFs.

When $\bfZ$ has bounded components, we have $\varphi_{\bfZ}(\bfx)=1$ in a neighborhood of the origin. Finally, the max-CF of $\max(\bfZ_1,\bfZ_2)$ (where the max is taken componentwise) evaluated at $\bfx$ is equal to the max-CF of the vector $(\bfZ_1,\bfZ_2)$ evaluated at the point $(\bfx,\bfx)$.

\begin{rem} \upshape When $d=1$, the max-CF of a nonnegative and integrable random variable $Z$ is
\begin{eqnarray*}
\varphi_Z(x) = {\rm E}\left\{\max\left(1,xZ\right)\right\}  &=& 1+\int_{1}^{\infty} \Pr(xZ > t) \, dt \\
													&=& 1+x\int_{1/x}^{\infty} \Pr\left(Z > z \right)\, dz \\
													&=& 1+x {\rm E}\{(Z-1/x)\ind_{\{ Z>1/x\}}\}.
\end{eqnarray*}
The latter expression is connected to the expected shortfall of $Z$; see \citet{embkm97}. Indeed, if $q_Z$ is the quantile function of $Z$ then the expected shortfall of $Z$ is defined, for all $\alpha \in (0,1)$, by
$$
\mathrm{ES}_Z(\alpha)=\frac{1}{1-\alpha}\int_{\alpha}^1 q_Z(\beta)\, d\beta.
$$
When the distribution function (df) of $Z$ is continuous, defining
$$
g(\beta)=\min\left(\frac{\beta}{1-\alpha},1 \right) = \left\{
    \begin{array}{ll}
        \dfrac{\beta}{1-\alpha} & \mbox{if} \quad \beta \leq 1-\alpha, \\[5pt]
        1 & \mbox{otherwise,}
    \end{array}
\right.
$$
for all $\beta\in (0,1)$, then
$$
\mathrm{ES}_Z(\alpha)=\frac{1}{1-\alpha}\int_0^{1-\alpha} q_Z(1-\beta)\, d\beta=\int_0^1 q_Z(1-\beta)\, dg(\beta).
$$
An integration by parts and the change of variables $\beta=\Pr(Z>z)$ give
\begin{eqnarray*}
\mathrm{ES}_Z(\alpha)=\int_0^1 g(\beta) dq_Z(1-\beta) &=& \int_0^{\infty} g\{\Pr(Z>z)\} dz \\
													  &=& q_Z(\alpha)+\frac{1}{1-\alpha} \int_{q_Z(\alpha)}^{\infty} \Pr\left(Z > z \right) \, dz. 											
\end{eqnarray*}
A similar argument in the more general context of Wang distortion risk measures is given in \citet{elmstu2016}. Letting $x=x_{\alpha}=1/q_Z(\alpha)$, $\alpha\in (0,1)$, we obtain
$$
\varphi_Z(x_{\alpha}) = 1+x_{\alpha} (1-\alpha) \{\mathrm{ES}_Z(\alpha) - q_Z(\alpha)\}.
$$
If the stop-loss premium risk measure of $Z$ is defined as
$$
\mathrm{SP}_Z(\alpha) = (1-\alpha)\{\mathrm{ES}_Z(\alpha)- q_Z(\alpha)\} = \int_{q_Z(\alpha)}^{\infty} \Pr\left(Z > z \right) \, dz,
$$
see \citet{embkm97}, then
$$
\varphi_Z(x_{\alpha}) = 1+x_{\alpha} \mathrm{SP}_Z(\alpha).
$$

\end{rem}
This remark suggests that max-CFs are closely connected to well-known elementary objects such as conditional expectations and risk measures; a particular consequence of it is that computing a max-CF is, in certain cases, much easier than computing a standard CF, i.e., a {\it  Fourier transform}. The following example illustrates this idea.

\begin{exam}\upshape Let $Z$ be a random variable having the generalized Pareto distribution with location parameter $\mu\geq 0$, scale parameter $\sigma>0$ and shape parameter $\xi\in (0,1)$, whose distribution function is
$$
\Pr(Z\leq z)=1-\left( 1+\xi\frac{z-\mu}{\sigma} \right)^{-1/\xi},\qquad z\ge \mu.
$$
The expression of the characteristic function of this distribution is a fairly involved one which depends on the Gamma function evaluated in the complex plane. However, it is straightforward that, for all $x>0$,
$$
\int_x^{\infty} \Pr\left(Z > z \right)\, dz = \begin{cases} {\rm E}(Z)-x=\mu-x+\dfrac{\sigma}{1-\xi} & \mbox{if } x <\mu, \\[10pt] \dfrac{\sigma}{1-\xi} \left( 1+\xi\dfrac{x-\mu}{\sigma} \right)^{1-1/\xi} & \mbox{if } x \geq\mu. \end{cases}
$$
Hence the max-CF of $Z$ is
$$
\varphi_Z(x)=\begin{cases} x{\rm E}(Z)=x\left( \mu+\dfrac{\sigma}{1-\xi} \right) & \mbox{if } x >\dfrac{1}{\mu}, \\[10pt] 1+\dfrac{\sigma x}{1-\xi} \left( 1+\xi\dfrac{1-\mu x}{\sigma x} \right)^{1-1/\xi} & \mbox{if } x \leq\dfrac{1}{\mu}. \end{cases}
$$
\end{exam}

 The following example is a consequence of the Pickands--de Haan--Resnick representation of a max-stable distribution function; see,  e.g., \citet[Theorems~4.2.5, 4.3.1]{fahure10}. In this paper, all operations on vectors $\bfx,\bfy\in\R^d$ such as $\bfx+\bfy$, $\bfx/\bfy$, $\bfx\le \bfy$, $\max(\bfx,\bfy)$ etc. are always meant componentwise.

\begin{exam}\upshape
Let $G$ be a $d$-dimensional max-stable distribution function with identical univariate Fr\'{e}chet-margins $G_i(x)=\exp(-x^{-\alpha})$, $x>0$, $\alpha >1$. Then there exists a $D$-norm $\norm\cdot_D$ on $\R^d$ such that $G(\bfx)=\exp\left(-\norm{1/\bfx^\alpha}_D\right)$, $\bfx>\bfzero\in\R^d$. Let the random vector $\bfxi$ have distribution function $G$. Its max-CF is
\begin{align*}
\varphi_{\bfxi}(\bfx)&= 1 + \int_1^\infty 1-\exp\left(-\frac{\norm{\bfx^\alpha}_D}{y^\alpha}\right)\,dy\\
&=1+ \norm{\bfx^\alpha}_D^{1/\alpha} \int_{1/\norm{\bfx^\alpha}_D^{1/\alpha}} 1-\exp(-y^{-\alpha})\,dy,\qquad \bfx\ge \bfzero\in\R^d.
\end{align*}
\end{exam}

This paper is organized as follows. In Section \ref{sec:convergence_of_max-cf} we establish among others the fact that pointwise convergence of max-CFs is equivalent to convergence with respect to the Wasserstein distance. In Section \ref{subsec:general_remarks_on_max-cf} we list some general remarks on max-CFs. In particular, it is shown that the space of max-CFs is not closed in the sense of pointwise convergence.
An inversion formula for max-CF, by which the distribution function of a nonnegative and integrable random variable can be restored by knowing its max-CF, is established in Section \ref{subsec:an_inversion_formula}.

\section{Convergence of max-characteristic functions}\label{sec:convergence_of_max-cf}

Denote by $d_W(P,Q)$ the Wasserstein metric between two probability
distributions on $\R^d$ with finite first moments, i.e.,
\begin{equation*}
  d_W(P,Q)\\
  =\inf\set{{\rm E}\left(\norm{ \bfX- \bfY}_1\right):\,  \bfX\mathrm{\ has\ distribution\ }P,\, \bfY \mathrm{\ has\ distribution\ }Q}.
\end{equation*}
It is well known that convergence of probability measures $P_n$ to $P_0$ with respect to the
Wasserstein metric is equivalent to weak convergence together with
convergence of the sequence of moments
\[
\int_{\R^d} \norm{\bfx}_1\,P_n(d \bfx) \to_{n\to\infty} \int_{\R^d}\norm{\bfx}_1\,P_0(d \bfx);
\]
see, for example, Definition 6.8 of \citet{villani09}.

Let $\bfX,\bfY$ be integrable random vectors in $\R^d$ with distributions $P$ and $Q$. By $d_W(\bfX,\bfY)=d_W(P,Q)$ we denote the Wasserstein distance between $\bfX$ and $\bfY$. The next result states that pointwise convergence of max-CFs is equivalent to convergence with respect to the Wasserstein metric.

\begin{theorem}\label{theo:characterization_of_pointwise_convergence_of_cf}
Let $\bfZ$, $\bfZ^{(n)}$, $n\in\N$, be nonnegative and integrable random vectors in $\R^d$ with corresponding max-CF $\varphi_{\bfZ}$, $\varphi_{\bfZ^{(n)}}$, $n\in\N$. Then $\varphi_{\bfZ^{(n)}}\to_{n\to\infty}\varphi_{\bfZ}$ pointwise $\Leftrightarrow$ $d_W\left(\bfZ^{(n)},\bfZ\right)\to_{n\to\infty}0$.
\end{theorem}

\begin{proof}
 Suppose that $d_W(\bfZ^{(n)},\bfZ)\to_{n\to\infty}0$. Then we can find versions $\bfZ^{(n)}$, $\bfZ$ such that ${\rm E}\left(\norm{\bfZ^{(n)}-\bfZ}_1\right)\to_{n\to\infty}0$. This implies, for $\bfx=(x_1,\dots,x_d)\ge 0$,
\begin{align*}
\varphi_{\bfZ^{(n)}}(\bfx) &= {\rm E}\left(\max [1,x_1 \{Z_1+(Z_1^{(n)}-Z_1) \}, \dots, x_d \{Z_d+(Z_d^{(n)}-Z_d) \} ]\right)\\
&\begin{cases}
\le {\rm E}\{\max(1,x_1Z_1,\dots,x_dZ_d) \} + \norm{\bfx}_\infty {\rm E}\left(\norm{\bfZ^{n}-\bfZ}_1\right)\\
\ge {\rm E}\left\{\max(1,x_1Z_1,\dots,x_dZ_d) \right\} - \norm{\bfx}_\infty {\rm E}\left(\norm{\bfZ^{n}-\bfZ}_1\right)
\end{cases}\\
&=\varphi_{\bfZ}(\bfx)+ o(1).
\end{align*}

Suppose next that $\varphi_{\bfZ^{(n)}}\to_{n\to\infty}\varphi_{\bfZ}$ pointwise. We have for $t>0$ and $\bfx=(x_1,\dots,x_d)\ge \bfzero$
$$
t\varphi_{\bfZ^{(n)}} \left (\frac \bfx t \right) = \EE \{ \max ( t,x_1 Z_1^{(n)},\ldots,x_d Z_d^{(n)})  \}.
$$
This gives
\begin{eqnarray*}
t\varphi_{\bfZ^{(n)}}\left( \frac \bfx t \right) &=& \int_0^{+\infty} \Pr \{ \max ( t,x_1 Z_1^{(n)},\ldots,x_d Z_d^{(n)} )>y  \} dy \\
									  &=& t+\int_t^{+\infty} \Pr \{\max ( x_1 Z_1^{(n)},\ldots,x_d Z_d^{(n)}  )>y \} dy
\end{eqnarray*}
so that
\begin{equation}\label{eqn:integral_representation_of_max_cf}
t\varphi_{\bfZ^{(n)}}\left(\frac \bfx t\right) = t + \int_t^\infty 1- \Pr (x_iZ_i^{(n)}\le y,1\le i\le d)\,dy.
\end{equation}
Now, for $\varepsilon > 0$ and $1\le i\le d$
\begin{align*}
{\rm E} (Z_i^{(n)} ) - {\rm E}\left(Z_i\right)&= \int_0^\infty 1- \Pr (Z_i^{(n)}\le y )\,dy - \int_0^\infty 1- \Pr\left(Z_i\le y\right)\,dy\\
&= \int_{\varepsilon/2}^\infty 1- \Pr (Z_i^{(n)}\le y )\,dy - \int_{\varepsilon/2}^\infty 1- \Pr\left(Z_i\le y\right)\,dy + R_{n,i}(\varepsilon)
\end{align*}
where
$$
|R_{n,i}(\varepsilon)| =\left| \int_0^{\varepsilon/2} \Pr (Z_i^{(n)}\le y ) - \Pr\left(Z_i\le y\right) \,dy \right| \leq \varepsilon/2.
$$
Equation~(\ref{eqn:integral_representation_of_max_cf}) then gives
$$
\left| {\rm E} (Z_i^{(n)} ) - {\rm E}\left(Z_i\right) \right| \leq \varepsilon \ \mbox{ for large enough } n,
$$
which entails convergence of ${\rm E} (Z_i^{(n)})$ to ${\rm E}\left(Z_i\right)$. Consequently, we have to establish weak convergence of $\bfZ^{(n)}$ to $\bfZ$. From Equation \eqref{eqn:integral_representation_of_max_cf} we obtain for $0<s<t$ and $\bfx=(x_1,\dots,x_d)\ge 0$
\begin{align}\label{eqn:limit_of_difference_of_mcf}
t \varphi_{\bfZ^{(n)}}\left(\frac \bfx t\right) - s \varphi_{\bfZ^{(n)}}\left(\frac \bfx s\right) &= \int_s^t \Pr (x_iZ_i^{(n)}\le y, 1\le i\le d )\,dy\nonumber\\
&\to_{n\to\infty} t \varphi_{\bfZ}\left(\frac \bfx t\right) - s \varphi_{\bfZ}\left(\frac \bfx s\right)\nonumber\\
&= \int_s^t \Pr\left(x_iZ_i\le y, 1\le i\le d\right)\,dy.
\end{align}

Let $\bfx=(x_1,\dots,x_d)\ge 0$ be a point of continuity of the distribution function of $\bfZ$. Suppose first that $\bfx>0$. Then we have
\[
\Pr (\bfZ^{(n)}\le \bfx ) = \Pr\left(\frac 1{x_i}Z_i^{(n)}\le 1,1\le i\le d\right).
\]
If
\[
\limsup_{n\to\infty} \Pr\left(\frac 1{x_i}Z_i^{(n)}\le 1,1\le i\le d\right) > \Pr\left(\frac 1 {x_i}Z_i\le 1,1\le i\le d\right)
\]
or
\[
\liminf_{n\to\infty} \Pr\left(\frac 1{x_i}Z_i^{(n)}\le 1,1\le i\le d\right) < \Pr\left(\frac 1 {x_i}Z_i\le 1,1\le i\le d\right),
\]
then Equation \eqref{eqn:limit_of_difference_of_mcf} readily produces a contradiction by putting $s=1$ and $t=1+\varepsilon$ or $t=1$ and $s=1-\varepsilon$ with a small $\varepsilon >0$. We, thus, have
\begin{equation}\label{eqn:weak_convergence_in_the_positive_case}
\Pr (\bfZ^{(n)}\le \bfx )\to_{n\to\infty} \Pr\left(\bfZ\le \bfx\right)
\end{equation}
for each point of continuity $\bfx=(x_1,\dots,x_d)$ of the distribution function of $\bfZ$ with strictly positive components.

Suppose next that $x_j=0$ for $j\in T\subset \set{1,\dots,d}$, $x_i>0$ for $i\not\in T$, $T\not=\emptyset$. In this case we have
\[
\Pr(\bfZ\le \bfx)= \Pr(Z_i\le x_i,i\not\in T, Z_j\le 0, j\in T)=0
\]
by the continuity from the left of the distribution function of $\bfZ$ at $\bfx$. We thus have to establish
\[
\limsup_{n\to\infty}\Pr (\bfZ^{(n)}\le \bfx ) =\limsup_{n\to\infty}\Pr\left(Z_i^{(n)}\le x_i,i\not\in T, Z_j^{(n)}\le 0, j\in T\right)=0.
\]
Suppose that
\[
\limsup_{n\to\infty}\Pr\left(Z_i^{(n)}\le x_i,i\not\in T, Z_j^{(n)}\le 0, j\in T\right) =c >0.
\]
Choose a point of continuity $\bfy>\bfx$. Then we obtain
\[
0<c\le \limsup_{n\to\infty}\Pr (\bfZ^{(n)}\le \bfy )=\Pr(\bfZ\le \bfy)
\]
by Equation \eqref{eqn:weak_convergence_in_the_positive_case}. Letting $\bfy$ converge to $\bfx$ we obtain $\Pr(\bfZ\le \bfx)\ge c>0$ and, thus, a contradiction. This completes the proof of Theorem \ref{theo:characterization_of_pointwise_convergence_of_cf}. \end{proof}

Convergence of a sequence of max-CFs is therefore stronger than the convergence of standard CFs: the example of a sequence of real-valued random variables $(Z_n)$ such that
$$
\Pr(Z_n=e^n)=\frac{1}{n} \ \mbox{ and } \ \Pr(Z_n=0)=1-\frac{1}{n}
$$
is such that $Z_n\to 0$ in distribution, as can be seen from computing the related sequence of CFs, but ${\rm E}(Z_n)=e^n/n\to \infty \neq 0$.

Corollary~\ref{corconv} below, which is obtained by simply rewriting Theorem~\ref{theo:characterization_of_pointwise_convergence_of_cf}, is tailored to applications to MEVT.

\begin{cor}\label{corconv}
Let $\bfX^{(n)}$, $n\in\N$, be independent copies of a random vector $\bfX$ in $\R^d$ that is nonnegative and integrable in each component. Let $\bfxi=(\xi_1,\dots,\xi_d)$ be a max-stable random vector with Fr\'{e}chet margins $\Pr(\xi_i\le x)=\exp\left(-1/x^{\alpha_i}\right)$, $x>0$, $\alpha_i>1$, $1\le i\le d$. Then we obtain from Theorem \ref{theo:characterization_of_pointwise_convergence_of_cf} the equivalence
\[
d_W\left(\frac{\max_{1\le i\le n}\bfX^{(i)}}{\bfa^{(n)}},\bfxi\right)\to_{n\to\infty} 0
\]
for some norming sequence $\bfzero<\bfa^{(n)}\in\R^d$ if and only if
\[
\varphi_n\to_{n\to\infty}\varphi_{\bfxi} \qquad \mbox{pointwise},
\]
where $\varphi_n$ denotes the max-CF of $\max_{1\le i\le n}\bfX^{(i)}/\bfa^{(n)}$, $n\in\N$.
\end{cor}

The following example shows a nice application of the use of max-CFs to the convergence of the componentwise maxima of independent generalized Pareto random vector in the total variation distance.

\begin{exam}\upshape
Let $U$ be a random variable that is uniformly distributed on $(0,1)$ and let $\bfZ=(Z_1,\dots,Z_d)$ be the generator of a $D$-norm $\norm\cdot_D$ with the additional property that each $Z_i$ is bounded,  i.e., $Z_i\le c$, $1\le i\le d$, for some constant $c\ge 1$. We require that $U$ and $\bfZ$ are independent.

Then the random vector
\[
\bfV=(V_1,\dots,V_d)= \frac 1 {U^{1/\alpha}} (Z_1^{1/\alpha},\dots,Z_d^{1/\alpha} )
\]
with $\alpha>0$ follows a \textit{multivariate generalized Pareto distribution}; see, e.g., \citet{buihz08} or \citet[Chapter 5]{fahure10}. Precisely, we have for $\bfx\ge (c^{1/\alpha},\dots,c^{1/\alpha})\in\R^d$
\begin{equation*}
\Pr(\bfV\le \bfx)= \Pr\left(U\ge \max_{1\le i\le d}\frac{Z_i}{x_i^{\alpha}}\right)
= 1- {\rm E}\left(\max_{1\le i\le d}\frac{Z_i}{x_i^{\alpha}}\right)
= 1 - \norm{\frac 1{\bfx^\alpha}}_D.
\end{equation*}
Let now $\bfV^{(1)}, \bfV^{(2)},\dots$ be independent copies of $\bfV$ and put
\[
\bfY^{(n)}= \frac{\max_{1\le i\le n}\bfV^{(i)}}{n^{1/\alpha}}.
\]
Then we have for $\bfx>\bfzero\in\R^d$ and $n$ large
\begin{equation}\label{eqn:convergence_of_gpd}
\Pr (\bfY^{(n)}\le \bfx ) = \left(1- \norm{\frac 1{n\bfx^\alpha}}_D\right)^n \to_{n\to\infty}\exp\left(-\norm{\frac 1{\bfx^\alpha}}_D \right) = \Pr(\bfxi\le \bfx),
\end{equation}
where $\bfxi$ is a max-stable random vector with identical Fr\'{e}chet margins $\Pr(\xi_i\le x)=\exp(-1/x^\alpha)$, $x>0$.
Choose $\alpha>1$; in this case the components of $\bfV$ and $\bfxi$ have finite expectations. By
writing
\[
\varphi_{\bfY^{(n)}}(\bfx) = 1 +\int_1^\infty 1- \Pr (\bfY^{(n)}\le t/\bfx )\,dt
\]
and using Equation \eqref{eqn:convergence_of_gpd}, elementary arguments such as a Taylor expansion make it possible to show that the sequence of max-CF $\varphi_{\bfY^{(n)}}$ converges pointwise to the max-CF $\varphi_{\bfxi}$ of $\bfxi$. Since convergence with respect to the Wasserstein metric is equivalent to convergence in distribution, denoted by $\to_d$, together with convergence of the  moments, we obtain from Theorem \ref{theo:characterization_of_pointwise_convergence_of_cf} that in this example we actually have both $\bfY^{(n)}\to_d\bfxi$ and ${\rm E} (Y^{(n)}_i )\to_{n\to\infty} {\rm E}(\xi_i)=\Gamma(1-1/\alpha)$ for $1\le i\le d$.
\end{exam}

\begin{exam}\upshape
Let  $\bfU^{(1)},\bfU^{(2)},\dots$ be independent copies of the random vector $\bfU=(U_1,\dots,U_d)$, which follows a copula $C$ on $\R^d$,  i.e., each $U_i$ is uniformly distributed on $(0,1)$. It is well-known (see, e.g., \citet[Section 5.2]{fahure10}) that there exists a non-degenerate random vector $\bfeta=(\eta_1,\dots,\eta_d)$ on $(-\infty,0]^d$ such that
\begin{equation*}
\bfV^{(n)}= n\left(\max_{1\le j\le n}\bfU^{(j)}-\bfone\right) \to_d\bfeta
\end{equation*}
 if and only if  there exists a $D$-norm $\norm\cdot_D$ on $\R^d$ such that, for all $\bfx\le\bfzero\in\R^d$,
\begin{equation*}
\Pr (\bfV^{(n)}\le \bfx )\to_{n\to\infty} \exp\left(-\norm{\bfx}_D\right)= G(\bfx),
\end{equation*}
 or if and only if there exists a $D$-norm $\norm\cdot_D$ on $\R^d$ such that
\begin{equation*}
C(\bfu) = 1-\norm{\bfone-\bfu}_D + o\left(\norm{\bfone-\bfu}\right)
\end{equation*}
as $\bfu\to\bfone$, uniformly for $\bfu\in[0,1]^d$.

  We have for $1\le i\le d$
  \begin{equation*}
  {\rm E}\left\{n\left(1-\max_{1\le j\le n}U_i^{(j)} \right)\right\}=\frac n{n+1}\to_{n\to\infty} 1
  \end{equation*}
and, thus, we obtain from Theorem \ref{theo:characterization_of_pointwise_convergence_of_cf} the characterization
\begin{equation*}
\bfV^{(n)}\to_d\bfeta \Leftrightarrow d_W\left(\bfV^{(n)},\bfeta\right)\to_{n\to\infty} 0 \Leftrightarrow\varphi_{-\bfV^{(n)}}\to_{n\to\infty}\varphi_{-\bfeta}\mbox{ pointwise}.
\end{equation*}
For instance, when $d=2$, straightforward computations yield that $-{\bfeta}$ arises as a weak limit above if and only if it has a max-CF of the form
$$
\varphi_{{-\bfeta}}(\bfx)= 1+x_1 \exp(-1/x_1)+x_2 \exp(-1/x_2)- \frac 1{\norm{1/\bfx}_D} \exp(-\norm{1/\bfx}_D).
$$
\end{exam}

\begin{cor}
Let $\bfZ$, $\bfZ^{(n)}$, $n\in\N$, be generators of $D$-norms on $\R^d$. Then $\varphi_{\bfZ^{(n)}}\to_{n\to\infty}\varphi_{\bfZ}$ pointwise $\Leftrightarrow$ $\bfZ^{(n)}\to_d \bfZ$.
\end{cor}

Interestingly, the convergence of a sequence of max-CFs of generators of $D$-norms also implies pointwise convergence of the related $D$-norms.  We denote by $\norm\cdot_{D,\bfZ}$ that $D$-norm, which is generated by $\bfZ$.

\begin{cor}\label{coro:convergence_of_max-CF_implies_convergence_of_D-norms}
Let $\bfZ$, $\bfZ^{(n)}$, $n\in\N$, be generators of $D$-norms in $\R^d$ with respective max-CF $\varphi_{\bfZ}$, $\varphi_{\bfZ^{(n)}}$, $n\in\N$. Then the pointwise convergence $\varphi_{\bfZ^{(n)}}\to_{n\to\infty}\varphi_{\bfZ}$ implies $\| \cdot \|_{D,\bfZ^{(n)}}\to_{n\to\infty}\| \cdot \|_{D,\bfZ}$ pointwise.
\end{cor}

\begin{proof}
We have for $\bfx=(x_1,\dots,x_d)\ge \bfzero$
\begin{align*}
\norm{\bfx}_{D^{(n)}}&= {\rm E}\left\{\max_{1\le i\le d}\left(x_iZ_i^{(n)}\right) \right\} = {\rm E}\left[\max_{1\le i\le d}\left\{x_iZ_i + x_i\left(Z_i^{(n)}-Z_i\right)\right\} \right]\\
&= {\rm E}\left\{\max_{1\le i\le d}(x_iZ_i)\right\} + O\left\{{\rm E}\left(\norm{\bfZ^{(n)}-\bfZ}_1  \right)\right\}\\
&\to_{n\to\infty}  {\rm E}\Bigl\{\max_{1\le i\le d}(x_iZ_i)\Bigr\} = \norm{\bfx}_D
\end{align*}
with proper versions of $\bfZ^{(n)}$ and $\bfZ$.
\end{proof}

\subsection{Some general remarks on max-characteristic functions}\label{subsec:general_remarks_on_max-cf}
The goal of this section is to give a few elements about the structure of the set of max-characteristic functions. This is done by constructing a particular functional mapping between max-CFs for generators of $D$-norms, and then iterating this mapping to draw our conclusions. Specifically, in what follows we let, for any $p\in (0,1]$, $T_p$ be the functional mapping which sends any function $f:\R^d\to\R$ to
$$
T_p(f)=1-p+p f\left( \frac{\cdot}{p} \right).
$$
\begin{lemma}
\label{maxCFtrans}
If $\varphi$ is the max-CF of a generator of a $D$-norm then, for any $p\in (0,1]$, so is the function $T_p(\varphi)$.
\end{lemma}

\begin{proof}
Let $\bfZ$ be a generator of the max-CF $\varphi$. Pick a Bernoulli random variable $U$ having expectation $p$ and independent of $\bfZ$, and set
$$
\psi(\bfx)={\rm E}\left\{\max\left( 1,x_1 \frac{U}{p} Z_1,\ldots,x_d \frac{U}{p} Z_d \right) \right\}.
$$
Then clearly $\psi$ is the max-CF of the generator of a $D$-norm, and
\begin{eqnarray*}
\psi(\bfx) &=& \EE\left\{ \ind_{\{ U=0 \}} + \max\left( 1,\frac{x_1}{p} Z_1,\ldots,\frac{x_d}{p} Z_d \right) \ind_{\{ U=1 \}} \right\} \\
		&=& \Pr(U=0) + \Pr(U=1) {\rm E}\left\{ \max\left( 1,\frac{x_1}{p} Z_1,\ldots,\frac{x_d}{p} Z_d \right) \right\}
\end{eqnarray*}
by the independence of $U$ and $Z$. The result follows because of the right-hand side being exactly $1-p+p\varphi(\bfx/p)$.
\end{proof}

\begin{lemma}
\label{maxCFiter}
For any integer $k\geq 1$, the $k$th iterate of the functional $T_p$ is
$$
f\mapsto T_p^{(k)}(f)=\Pr(X\leq k)+ p^k f\left( \frac{\cdot}{p^k} \right),
$$
where $X$ is a geometric random variable having parameter $1-p$.
\end{lemma}

\begin{proof}
The result is clearly true for $k=1$. That the conclusion holds for every integer $k$ follows by straightforward induction because
$$
(1-p)+p \Pr(X\leq k)=(1-p)+p\sum_{j=1}^k p^{j-1} (1-p) = \sum_{j=1}^{k+1} p^{j-1} (1-p)= \Pr(X\leq k+1)
$$
whenever $X$ has a geometric distribution with parameter $1-p$.
\end{proof}

In the following lemma, the phrase ``$\bfx\to \infty \mbox{ in } \RR_+^d$'' means $\norm{\bfx}_\infty\to\infty$ and $\bfx\in  \RR_+^d$.
\begin{lemma}
\label{maxCFbound}
If $\varphi_{\bfZ}$ is the max-CF of a generator $\bfZ$ of a $D$-norm, then
$$
\max(1,\| \bfx \|_{D,Z})\leq \varphi_{\bfZ}(\bfx)\leq 1+\| \bfx \|_{D,\bfZ} \ \mbox{ for all } \ \bfx=(x_1,\ldots,x_d)\in \RR_+^d.
$$
Especially, if $\mathcal{G}$ denotes the set of all generators of $D$-norms,
$$
\sup_{\bfZ\in \mathcal{G}} \left| \frac{\varphi_{\bfZ}(\bfx)}{\| \bfx \|_{D,\bfZ}} - 1 \right| \to 0 \ \mbox{ as } \ \bfx\to \infty \mbox{ in } \RR_+^d.
$$
\end{lemma}

\begin{proof}
The lower bound is obtained by noting that
\begin{equation}
1\leq \max(1,x_1 Z_1,\ldots,x_d Z_d), \; \ \max(x_1 Z_1,\ldots,x_d Z_d)\leq \max(1,x_1 Z_1,\ldots,x_d Z_d)
\end{equation}
and taking expectations. The upper bound is a consequence of the inequality $\max(a,b)\leq a+b$, valid when $a,b\geq 0$. Finally, the uniform convergence result is obtained by writing
$$
1\leq \frac{\varphi_{\bfZ}(\bfx)}{\| \bfx \|_{D,\bfZ}}\leq 1+\frac{1}{\| \bfx \|_{D,\bfZ}}  \ \mbox{ for all } \ \bfZ\in \mathcal{G} \mbox{ and } \bfx\in \RR_+^d\setminus\{ \bfzero \}.
$$
Because $\| \cdot \|_{D,\bfZ} \geq \| \cdot \|_{\infty}$, this entails
$$
\sup_{\bfZ\in \mathcal{G}} \left| \frac{\varphi_{\bfZ}(\bfx)}{\| \bfx \|_{D,\bfZ}} - 1 \right| \leq \frac{1}{\| \bfx \|_{\infty}} \ \mbox{ for all } \ \bfx\in \RR_+^d\setminus\{ \bfzero \}
$$
from which the conclusion follows.
\end{proof}

It is noteworthy that the inequalities of Lemma \ref{maxCFbound} are sharp, in the sense that for $\bfZ=(1,\ldots,1)$, $\varphi_{\bfZ}(\bfx) = \max(1,\| \bfx \|_{\infty}) =\max(1,\| \bfx \|_{D,Z})$ and therefore the leftmost inequality is in fact an equality in this case, while the rightmost inequality $\varphi_{\bfZ}(\bfx)\leq a+b \| \bfx \|_{D,\bfZ}$ can only be true if $a,b\geq 1$ because of the leftmost inequality again.

Lemma \ref{maxCFbound} has the following corollary, which can also be obtained as a consequence of the monotone convergence theorem.

\begin{cor}
\label{maxCFrangecst}
No constant function can be the max-CF of a generator of a $D$-norm.
\end{cor}

Such a result is of course not true for standard CFs, since the CF of the constant random variable 0 is the constant function 1.

The next result looks at what can be said when examining the pointwise limit of iterates of the functional $T_p$ on the set of max-CFs.

\begin{prop}
\label{maxCFlim}
If $\varphi_{\bfZ}$ is the max-CF of a generator $\bfZ$ of a $D$-norm, then for any $p\in (0,1)$, the sequence of mappings $\{T_p^{(k)}(\varphi_{\bfZ})\}$ has a pointwise limit which is independent of $p$ and equal to
$$
T(\varphi_{\bfZ})=1+\| \cdot \|_{D,\bfZ}.
$$
\end{prop}

\begin{proof}
By Lemma~\ref{maxCFiter}, we have for any $\bfx\in \RR_+^d$, $\bfx\not=\bfzero\in\R^d$, and any $k\geq 1$ that
$$
T_p^{(k)}(\varphi)(\bfx)=\Pr(X\leq k)+ p^k \varphi_{\bfZ}\left( \frac{\bfx}{p^k} \right).
$$
On one hand, when $k\to\infty$, the first term on the right-hand side converges to 1; on the other hand, because $p\in (0,1)$, we have $\bfx/p^k\to\infty$ in $\RR_+^d$ and therefore
$$
\lim_{k\to\infty} p^k \varphi_{\bfZ}\left( \frac{\bfx}{p^k} \right) = \| \bfx \|_{D,\bfZ} \lim_{k\to\infty} \frac{\varphi_{\bfZ}(\bfx/p^k)}{\| \bfx/p^k  \|_{D,\bfZ}} =  \| \bfx \|_{D,\bfZ}
$$
by Lemma~\ref{maxCFbound}. The conclusion follows by adding these limits.
\end{proof}

\begin{cor}
\label{maxCFcoro1}
If $\| \cdot \|_{D,\bfZ}$ is any $D-$norm then there is an explicit, iterative way to realize the function $1+\| \cdot \|_{D,\bfZ}$ as a limit of max-CFs. In particular, the expression of a $D-$norm is explicitly determined by the knowledge of the max-CF of any of its generators.
\end{cor}

Note that this result certainly cannot be true the other way around, since a single $D-$norm can in general be generated by different generators.

The next result looks a bit further into the range of the map $\bfZ\mapsto \varphi_{\bfZ}$. By considering the generator $(1,\ldots,1)\in \RR^d$ that generates the $D$-norm $\norm\cdot_\infty$, it is obvious that $\max(1,\| \cdot \|_{\infty})$ is actually the max-CF of a generator of a $D$-norm. Looking at Lemma~\ref{maxCFbound}, one may wonder if this remains true if $\| \cdot \|_{\infty}$ is replaced by some other $D-$norm, or, in other words, if the lower bound $\max(1,\| \cdot \|_{D,Z})$ in Lemma~\ref{maxCFbound} can be achieved as a $D-$norm, and similarly for the upper bound $1+\| \cdot \|_{D,Z}$. The next result says that this is not the case.

\begin{prop}
\label{maxCFrange}
Let $\bfZ$ be a generator of a $D$-norm.

\begin{enumerate}
\item[(i)] The mapping $1+\| \cdot \|_{D,\bfZ}$ cannot be the max-CF of a generator of a $D$-norm.
\item[(ii)] If moreover $\| \cdot \|_{D,\bfZ}\neq \| \cdot \|_{\infty}$, then $\max(1,\| \cdot \|_{D,\bfZ})$ cannot be the max-CF of a generator of a $D$-norm.
\end{enumerate}
\end{prop}

\begin{proof}
We start by proving (i). Suppose there is  a generator of a $D$-norm $\bfY$ such that $\varphi_{\bfY}=1+\| \cdot \|_{D,\bfZ}$. By Proposition~\ref{maxCFlim}, the sequence of mappings $T_p^{(k)}(\varphi_{\bfY})$, $k\geq 1$, has the pointwise limit
$$
T(\varphi_{\bfY})=1+\| \cdot \|_{D,\bfY}.
$$
Besides, if $X$ is a geometric random variable with parameter $1-p$, then for all $\bfx\in \RR_+^d$
$$
T_p^{(k)}(\varphi_{\bfY})(\bfx)=T_p^{(k)}(1+\| \cdot \|_{D,\bfZ})(\bfx)=\Pr(X\leq k)+ p^k (1+\| \bfx/p^k \|_{D,\bfZ})\to 1+\| \bfx \|_{D,\bfZ}
$$
as $k\to\infty$, so that $\| \cdot \|_{D,\bfY} = \| \cdot \|_{D,\bfZ}$.

We now conclude by using Theorem~\ref{theo:characterization_of_pointwise_convergence_of_cf}: the random vector
$$
\bfY^{(n)} = \frac{U_1 \cdots U_n}{p^n} \bf Y
$$
where $U_1,\ldots,U_n$ are independent Bernoulli random variables with mean $p$ which are independent of $\bf Y$, is the generator of a $D$-norm, with max-CF
$$
\varphi_{\bfY^{(n)}} = T_p^{(n)}(\varphi_{\bfY});
$$
see the proof of Lemma~\ref{maxCFtrans} and Lemma~\ref{maxCFiter}. By Proposition~\ref{maxCFlim}, $\varphi_{\bfY^{(n)}}\to_{n\to\infty}1+\| \cdot \|_{D,\bfY}=\varphi_{\bfY}$ pointwise, and thus Theorem~\ref{theo:characterization_of_pointwise_convergence_of_cf} yields $d_W\left(\bfY^{(n)},\bfY\right)\to_{n\to\infty}0$.
But
$$
\Pr (\bfY^{(n)} \neq {\bf 0} )=p^n\to_{n\to\infty}0
$$
which shows that $\bfY^{(n)}$ converges in distribution to $\bfzero$. This is a contradiction and (i) is proven.

We turn to the proof of (ii). Again, suppose there is a generator of a $D$-norm $\bfY$ such that $\varphi_{\bfY}=\max(1,\| \cdot \|_{D,\bfZ})$. We shall prove that $\| \cdot \|_{D,\bfZ}=\| \cdot \|_{\infty}$. The sequence of mappings $T_p^{(k)}(\varphi_{\bfY})$, $k\geq 1$, has the pointwise limit
$$
T(\varphi_{\bfY})=1+\| \cdot \|_{D,\bfY},
$$
and if $X$ is a geometric random variable with parameter $1-p$ then for all $\bfx\in \RR_+^d$,
\begin{align*}
T_p^{(k)}(\varphi_{\bfY})(\bfx)&=T_p^{(k)}\{\max(1,\| \cdot \|_{D,\bfZ})\}(\bfx)\\
&= \Pr(X\leq k)+ p^k \max(1,\| \bfx/p^k  \|_{D,\bfZ}) \\
&= \Pr(X\leq k)+ \max(p^k,\| \bfx  \|_{D,\bfZ}) \\
&\to  1+\| \bfx \|_{D,\bfZ}
\end{align*}
as $k\to\infty$, so that $\| \cdot \|_{D,\bfY} = \| \cdot \|_{D,\bfZ}$. Consequently:
$$
 \varphi_{\bfY}(\bfx)= {\rm E}\{\max(1,x_1 Y_1,\ldots,x_d Y_d)\} =\max[1,\EE\{\max(x_1 Y_1,\ldots,x_d Y_d)\}]
$$
for all $\bfx=(x_1,\ldots,x_d)\in \RR_+^d$. For all $i\in \{ 1,\ldots,d\}$, specializing $x_j=0$ for $j\neq i$ and $x_i=1$ gives
$$
\EE\{\max(1,Y_i)\} =\max\{1,\EE(Y_i)\} = 1 \ \mbox{ for all } \ i\in \{ 1,\ldots,d\}.
$$
Because $Y_i$ has expectation 1, this implies that the random variables $\max(1,Y_i)-1$ and $\max(1,Y_i)-Y_i$, being nonnegative and having expectation zero, must be almost surely zero. In other words, $Y_i\leq 1$ and $Y_i\geq 1$ almost surely, and thus $Y_i=1$ almost surely for all $i$. But then $Y=(1,\ldots,1)$ is a generator of the norm $\| \cdot \|_{\infty}$, so that $\| \cdot \|_{D,\bfZ}=\| \cdot \|_{D,\bfY}=\| \cdot \|_{\infty}$. The proof is complete.
\end{proof}

Combining Propositions~\ref{maxCFlim} and~\ref{maxCFrange}(i), we get the following corollary.

\begin{cor}
\label{maxCFcoro2}
The set of max-CFs of generators of $D$-norms is not closed in the sense of pointwise convergence.
\end{cor}

It should be noted that Corollary~\ref{maxCFcoro2} is also true for usual characteristic functions, as we can see with the example of a sequence of random variables $(X_n)$ such that for every $n$, $X_n$ is normally distributed, centered, and has variance $n^2$. Then
$$
\varphi_n(t)=\EE\left(e^{itX_n} \right) = e^{-n^2 t^2/2} \ \mbox{ for all } \ t\in\RR
$$
so that the sequence $(\varphi_n)$ converges pointwise to the indicator function of $\{ 0 \}$, which is not a characteristic function because it is not continuous.

\subsection{An inversion formula for max-characteristic functions}\label{subsec:an_inversion_formula}

As mentioned in the Introduction, any max-CF is a convex function and thus it is continuous and almost everywhere differentiable; furthermore, its derivative from the right exists everywhere.

Recall that for a vector $\bfx\in \RR^d$, the notation $\bfx>\bfzero$ means that $\bfx$ has strictly positive components. The next result contains both an inversion formula for max-CFs and a criterion for a function to be a max-CF.

\begin{prop}
\label{invmaxCF}
Let $\bfZ$ be a nonnegative and integrable random vector with max-CF $\varphi_{\bfZ}$.
\begin{enumerate}
\item[(i)] We have, for all $\bfx=(x_1,\ldots,x_d)>\bfzero$,
$$
\Pr(Z_j\leq x_j, \ 1\leq j\leq d)= \frac{\partial_{+}}{\partial t} \left\{ t\varphi_{\bfZ}\left( \frac{1}{t\bfx} \right) \right\}\Bigr|_{\substack{t=1}}
$$
where $\partial_{+}/\partial t$ denotes the right derivative with respect to the univariate variable $t$.
\item[(ii)] If $\psi$ is a continuously differentiable function such that
\begin{eqnarray*}
\frac{\partial}{\partial t} \left\{ t\psi\left( \frac{1}{t\bfx} \right) \right\}\Bigr|_{\substack{t=1}} &=& \Pr(Z_j\leq x_j, \ 1\leq j\leq d) \\
\mbox{and } \ \lim_{t\to\infty} t\left\{ \psi\left( \frac{1}{t\bfx} \right) -1 \right\} &=& 0
\end{eqnarray*}
for all $\bfx=(x_1,\ldots,x_d)>\bfzero$, then $\psi=\varphi_{\bfZ}$ on $(0,\infty)^d$.
\end{enumerate}
\end{prop}

\begin{proof}
Notice first that, similarly to equation~(\ref{eqn:integral_representation_of_max_cf}), we have
$$
t\varphi_{\bfZ}\left( \frac{1}{t\bfx} \right) = t+\int_t^{+\infty} 1-\Pr(Z_j\leq yx_j, \ 1\leq j \leq d)\, dy.
$$
Note that the above representation yields $\lim_{t\to\infty}t[\varphi_{\bfZ}\{1/(t\bfx)\}-1] =0$.

To show (i), notice that taking right derivatives with respect to $t$ yields
$$
\frac{\partial_{+}}{\partial t} \left\{t\varphi_{\bfZ}\left( \frac{1}{t\bfx} \right) \right\} = \Pr(Z_j\leq tx_j, \ 1\leq j\leq d).
$$
Setting $t=1$ concludes the proof of (i). To prove (ii), remark that
$$
\frac{\partial}{\partial t} \left\{ t\psi\left( \frac{1}{t\bfx} \right) \right\} = \psi\left( \frac{1}{t\bfx} \right) - \frac{1}{t}\sum_{i=1}^d  \frac{1}{x_j} \partial_j \psi\left( \frac{1}{t\bfx} \right) \ \mbox{ for all } \ t>0,
$$
where $\partial_j \psi$ denotes the partial derivative of $\psi$ with respect to its $j$th component. In particular, because
$$
\Pr(Z_j\leq x_j, \ 1\leq j\leq d) = \frac{\partial}{\partial t} \left\{ t\psi\left( \frac{1}{t\bfx} \right) \right\}\Bigr|_{\substack{t=1}} = \psi\left( \frac{1}{\bfx} \right) - \sum_{i=1}^d  \frac{1}{x_j} \partial_j \psi\left( \frac{1}{\bfx} \right)
$$
we obtain by replacing $\bfx$ with $t\bfx$ that for all $t>0$,
$$
\frac{\partial}{\partial t} \left\{ t\psi\left( \frac{1}{t\bfx} \right) \right\} = \Pr(Z_j\leq t x_j, \ 1\leq j\leq d).
$$
Write now
\begin{eqnarray*}
t\psi\left( \frac{1}{t\bfx} \right) &=& t -\int_t^{\infty} \frac{\partial}{\partial y} \left[ y\left\{ \psi\left( \frac{1}{y\bfx} \right) -1 \right\} \right] dy \\
									&=& t +\int_t^{\infty} 1-\Pr(Z_j\leq yx_j, \ 1\leq j \leq d)\, dy \\
									&=& t\varphi_{\bfZ}\left( \frac{1}{t\bfx} \right)
\end{eqnarray*}
to conclude the proof of (ii).
\end{proof}

\begin{rem}\upshape
This result makes it possible to improve upon the result of Proposition~\ref{maxCFrange}(i). Assume that $\varphi_{\bfZ}$ is the max-CF of a nonnegative and integrable random vector such that
$$
\varphi_{\bfZ}(\bfx)=\psi(1,\| \bfx \|),
$$
where $\psi:\RR_+^2\to \RR_+$ is a 1-homogeneous function and $\| \cdot \|$ is a norm on $\RR^d$. Then informally,
$$
\frac{\partial_{+}}{\partial t} \left\{ t\varphi_{\bfZ}\left( \frac{1}{t\bfx} \right) \right\}\Bigr|_{\substack{t=1}} = \frac{\partial_{+}}{\partial t}\left\{ \psi(t,\| 1/\bfx \|) \right\}\Bigr|_{\substack{t=1}} = \partial_{1,+}\psi(1,\| 1/\bfx \|)
$$
if $\partial_{1,+}$ denotes the right derivative with respect to the first component. In particular,
$$
\frac{\partial_{+}}{\partial t} \left\{ t\varphi_{\bfZ}\left( \frac{1}{t\bfx} \right) \right\}\Bigr|_{\substack{t=1}} \to \begin{cases} \partial_{1,+}\psi(1,0) & \mbox{if } \bfx\to \bfinfty, \\[5pt] \partial_{1,+}\psi(1,\infty) & \mbox{if } \bfx\to \bf0. \end{cases}
$$
In other words, by Proposition~\ref{invmaxCF}, unless $\partial_{1,+}\psi(1,y)$ both converges to 1 as $y\to 0$ and to 0 as $y\to\infty$, the function $\psi(1,\| \cdot \|)$ cannot be a max-CF. Applying this to the example $\psi(x,y)=x+y$, we find the result of Proposition~\ref{maxCFrange}(i) again.
\end{rem}

\section*{Acknowledgment} The authors are indebted to Professor  Chen Zhou for stimulating discussions, which led to Lemma \ref{lem:uniqueness_of_max-cf} and the definition of a max-CF. The authors are also grateful to two anonymous reviewers for their careful reading of the manuscript and their constructive remarks.

\end{document}